\documentclass{amsart}
\usepackage{amsfonts,amssymb,amsmath,amsthm}
\usepackage{url}
\usepackage{enumerate}
\usepackage{graphicx}
\usepackage{hyperref}
\usepackage{graphicx}
\usepackage{siunitx}
\usepackage{float}
\restylefloat{figure}
\newtheorem{theorem}{Theorem}[section]
\newtheorem{lemma}[theorem]{Lemma}

\author{Purevsuren Damba         \and
        Uuganbaatar Ninjbat}
\address{
Mathematics Department\\
The National University of Mongolia\\
Ulaanbaatar, Mongolia}
\email{purevsuren@smcs.num.edu.mn \and uugnaa.ninjbat@gmail.com}

\keywords{Convexity, Intersection graph, Spherical polygon, Side disk}
\subjclass[2010]{52A10, 52C26, 05C10}
\begin{document}
\title[Side Disks of a Spherical Great Polygon]{Side Disks of a Spherical Great Polygon}
\begin{abstract}
Take a circle and mark $n\in\mathbb{N}$ points on it designated as vertices. For any arc segment between two consecutive vertices which does not pass through any other vertex, there is a disk centered at its midpoint and has its end points in the boundary. We analyze intersection behaviour of these disks and show that the number of disjoint pairs among them is between $\frac{(n-2)(n-3)}{2}$ and $\frac{n(n-3)}{2}$ and their intersection graph is a subgraph of a triangulation of a convex $n$-gon.
\end{abstract}
\maketitle

\section{Introduction}
\label{intro}
Huemer and P\'{e}rez-Lantero \cite{huemer} studied intersection behaviour of disks with the sides of a convex $n$-gon as their diameters, which are called as side disks, and showed that their intersection graph is planar; see Theorem 4 in \cite{huemer}. Recall that {\em intersection graph} of a set of figures is a graph in which each vertex represents one and only one of those figures and they are adjacent if and only if the corresponding figures are intersecting. This result has a direct combinatorial consequence: the number of disjoint pairs among these disks is at least $\frac{(n-3)(n-4)}{2}$, which follows from the fact that every planar graph with $n$ vertices has at most $3(n-2)$ edges; see Corollary 11.1(b) in \cite{harary}. 

We believe that the problem of analyzing intersection patterns of side disks is of considerable interest because of the geometrical challenges resulting from its unusual conclusion, i.e. it reflects on disjointness of geometrical figures. In Euclidean geometry a search for new results by replacing line segments with conic sections is often rewarding as illustrated in the following well known results: Pappus's hexagon theorem vs. Pascal's theorem, and Ceva's theorem vs. Haruki's theorem (see Chap. 6 in \cite{chamberland}). Accordingly, in Sect. \ref{circle} instead of a convex $n$-gon we consider a circle partitioned into $n\in\mathbb{N}$ arc segments. The concept of side disk naturally extends to this setting: for each arc segment there is a unique disk centered at its midpoint, and is having its two end points on its boundary. When $n=5$ the resulting configuration already appears in Miquel's five circles theorem (see Chap. 5 in \cite{chamberland}). In Theorem \ref{circle-thm} we show that for $n\geq 3$ there are at least $\frac{(n-2)(n-3)}{2}$ and at most $\frac{n(n-3)}{2}$ disjoint pairs of side disks for the partitioned circle with $n$ arc segments. We also verify that these bounds are tight for all $n\geq 3$, and the intersection graph of these disks is a subgraph of a triangulation of a convex polygon (see Theorem \ref{planar}). 

Throughout this paper we use the following conventions. For any points $X$, $Y$ and $Z$ in the plane, the line passing through $X, Y$ is denoted as $XY$-line, their connecting line segment is denoted as $XY$, and $|XY|$ is its length. $\measuredangle XYZ$ is the angle between $XY$ and $YZ$ measured in the clockwise direction. For any disk $\omega$, its boundary circle is denoted as $\partial(\omega)$ and when there is no ambiguity, we identify a given disk (or circle) with its center $X$ and call it $X$-disk (or $X$-circle), etc. For any plane regions $\omega$ and $\tau$, $(\omega\cap\tau)$ is the region in their intersection, and $\omega\subset\tau$ means the former is included (strictly) in the latter, i.e. every point in $\omega$ is in $\tau$ but not vice versa. For a point $X$ and region $\tau$, $X\in \tau$ means $X$ is located in $\tau$ and $X\notin \tau$ means the opposite.

\section{The main results}
\label{circle}
Let $C_{n}$ be a circle partitioned into $n\in\mathbb{N}$ arc segments by marking $n$ points on it. We identify each marked point as vertex and each arc segment between two consecutive vertices which does not pass through any other vertex as a side. Then, $C_{n}$ is a spherical polygon with vertices at a great circle and we call it as {\em spherical great polygon}; for more on spherical polygons see e.g. Chap. 6.4 in \cite{princeton}. The case where each side has the same length is denoted as $C_{n}^{\star}$. For each side of $C_{n}$, there is a unique disk centered at its midpoint and is having its two end points on the boundary. This is the {\em side disk} of that side and two side disks are {\em neighbouring} if their corresponding sides are adjacent.

Notice that each side of $C_{n}$ divides its disk into two parts, one of which intersects with the region enclosed by $C_{n}$. We call this as {\em inner part} and the other as {\em outer}, and as a convention we include the corresponding arc of $C_{n}$ to the inner part of the side disk, but not to the outer. Then, convexity implies that outer parts of two side disks of $C_{n}$ do not intersect. We shall prove two lemmas. 
\begin{lemma}
\label{fundament}
Let $\omega$ be a given disk and $A$, $B$, $C$ be points on $\partial(\omega)$ such that $AC$-arc is a segment of $AB$-arc. If $\omega_{1}$ and $\omega_{2}$ are the side disks of $AB$-arc and $AC$-arc, respectively, then $(\omega\cap\omega_{2})\subset(\omega\cap\omega_{1})$ and any point in $(\omega\cap\omega_{2})$ except $A$ is in the interior of $\omega_{1}$.
\end{lemma}
\begin{proof}
Let $O_{1}$ and $O_{2}$ be the centers of $\omega_{1}$ and $\omega_{2}$, respectively. Since $AC$-arc is contained in $AB$-arc, $O_{2}$ must be on the $AO_{1}$-arc not passing through $B$ (see Fig. \ref{funda}). Since $O_{1}$ is the mid-point of $AB$-arc, $AO_{1}$-arc is always less than a half of $\partial(\omega)$. Thus, $\measuredangle O_{1}O_{2}A>\frac{\pi}{2}$ and $\triangle AO_{2}O_{1}$ is an obtuse triangle with $|AO_{1}|>|AO_{2}|$. Let $O_{3}$ be the point on $AO_{1}$ with $|AO_{3}|=|AO_{2}|$, and $\omega_{3}$ be the disk centered at $O_{3}$ and is having $A$ on its boundary (see the dashed disk in Fig. \ref{funda}). Since $A$, $O_{3}$ and $O_{1}$ are collinear and $|AO_{1}|>|AO_{3}|$, we have $\omega_{3}\subset\omega_{1}$ and $A=\partial(\omega_{3})\cap \partial(\omega_{1})$. Then, we can conclude that $(\omega\cap\omega_{3})\subset(\omega\cap\omega_{1})$, and the only point in $(\omega\cap\omega_{3})$ which is on $\partial(\omega_{1})$ is $A$. 
\begin{figure}[h]
\centering
\includegraphics[height=1.5 in, keepaspectratio = true]{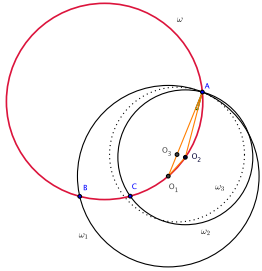}
\caption{Illustration for Lemma \ref{fundament}}
\label{funda}
\end{figure} 
On the other hand, $\omega_{2}$ is a rotation of $\omega_{3}$ around $A$ in the direction to move its center from an interior point of $\omega$, $O_{3}$, to a boundary point, $O_{2}$. Thus, we must have $(\omega\cap\omega_{2})\subset(\omega\cap\omega_{3})$, which implies $(\omega\cap\omega_{2})\subset(\omega\cap\omega_{3})\subset (\omega\cap\omega_{1})$. Finally, from our proof its clear that any point in $(\omega\cap\omega_{2})$ except $A$ must be in the interior of $\omega_{1}$.
\end{proof}
\begin{lemma}
\label{quad-out}
Let $\omega$ be a given disk and $A$, $B$, $C$, $D$ be four points marked subsequently on $\partial(\omega)$. Further let $\omega_{ab}$ be the side disk corresponding to the $AB$-arc, and let $\omega_{bc}$, $\omega_{cd}$ and $\omega_{da}$ be defined analogously (see Fig. \ref{quad-out-fig}). Let $X$ be the intersection point of $\partial(\omega_{da})$ and $\partial(\omega_{cd})$, other than $D$; and $Y$, $Z$ and $T$ be defined analogously for the pairs $\partial(\omega_{cd})$ and $\partial(\omega_{bc})$, $\partial(\omega_{bc})$ and $\partial(\omega_{ab})$, and $\partial(\omega_{ab})$ and $\partial(\omega_{da})$, respectively. Then,
\begin{itemize}
\item[(a)] $X, Y \notin \omega_{ab}$, $Y,Z\notin\omega_{da}$, $Z,T\notin\omega_{cd}$ and $X, T\notin\omega_{bc}$; and
\item[(b)] Quadrilateral $XYZT$ is a rectangle.
\end{itemize}
\end{lemma}
\begin{proof}
To prove Lemma \ref{quad-out} (a), it suffices to show that $Z\notin\omega_{cd}$ as similar argument applies to the others. Consider Fig. \ref{quad-out-fig} and let the dashed disk $\omega_{bd}$ be the disk corresponding to $BD$-arc. 
\begin{figure}[h]
\centering
\includegraphics[height=1.7 in, keepaspectratio = true]{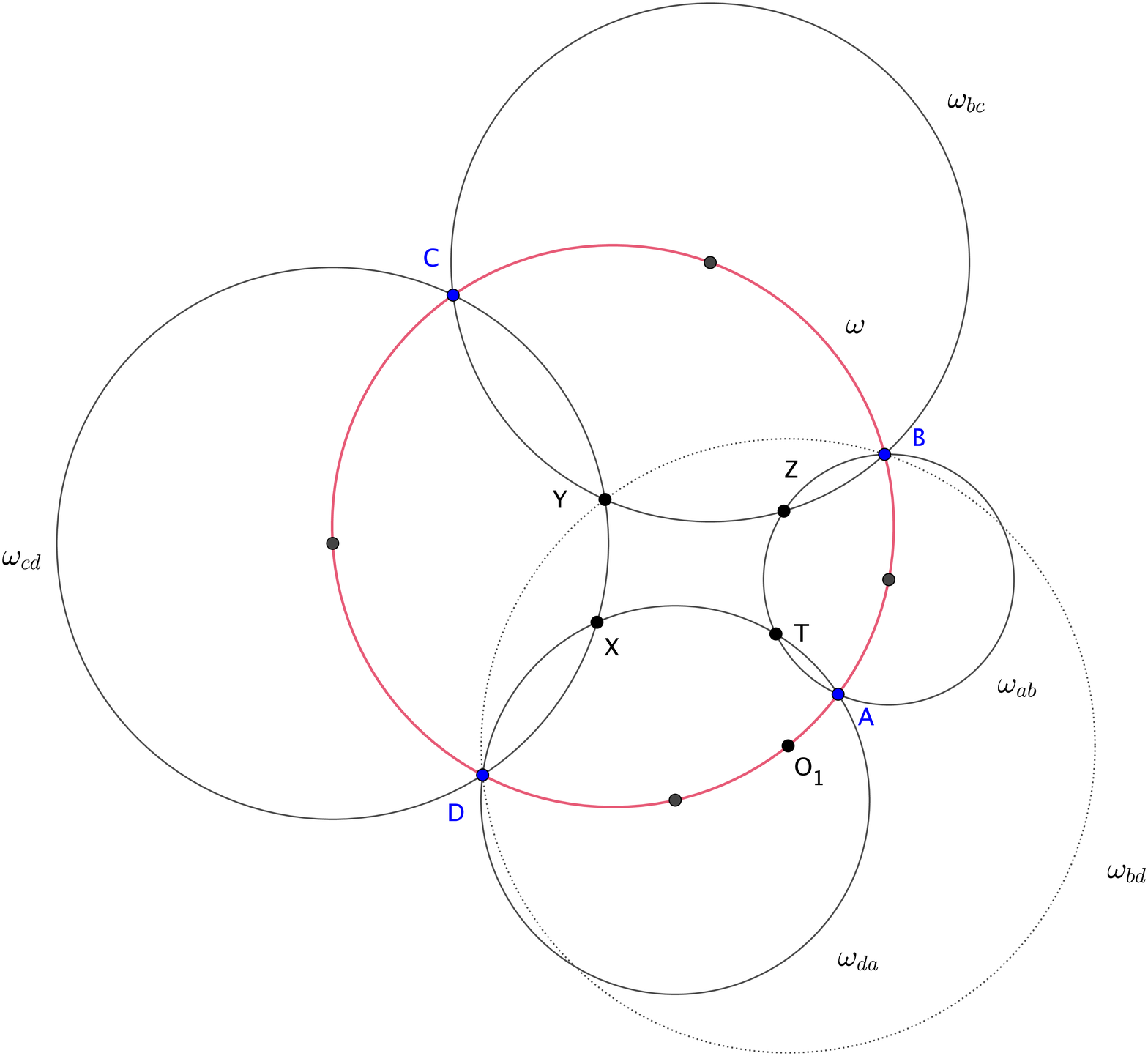}
\caption{Illustration for Lemma \ref{quad-out} (a)}
\label{quad-out-fig}
\end{figure}
It is well known, and can easily be proven that $\partial(\omega_{bd})$ passes through $Y$, which is the incenter of $\triangle BCD$ (see below). Notice that all conditions of Lemma \ref{fundament} are met for $\omega$, $\omega_{bd}$ and $\omega_{ab}$. Thus, $Z\in(\omega\cap\omega_{ab})$ must be located in the interior of $\omega_{bd}$. But then $Z\notin\omega_{cd}$ as the only point which is in $(\omega_{cd}\cap\omega_{bd}\cap\omega_{bc})$ is $Y$, and $Y$ and $Z$ are distinct points as $Y\in \partial(\omega_{bd})$ while $Z\notin \partial(\omega_{bd})$. This proves Lemma \ref{quad-out} (a).

Let $H$, $G$, $F$ and $W$ be the centers of $\omega_{ab}$, $\omega_{bc}$, $\omega_{cd}$ and $\omega_{da}$, respectively. We claim that $X$, $Y$, $Z$ and $T$ are the incenters of $\triangle ADC$, $\triangle DCB$, $\triangle CBA$ and $\triangle BAD$, respectively (see Fig. \ref{quad-rect-fig}). Notice that since $F$ and $W$ are the centers of two circles intersecting at $D$ and $X$, $FW$ is a perpendicular bisector of $DX$ and $\measuredangle XFD=2 \measuredangle WFD$. Since $W$ is the midpoint of $AD$-arc, we also have $\measuredangle WFD=\frac{1}{2} \measuredangle AFD$, which implies $\measuredangle XFD=\measuredangle AFD$. Thus, points $F$, $X$ and $A$ are collinear. Since $F$ is the midpoint of $DC$-arc, $\measuredangle DAF=\measuredangle FAC$, hence $AF$ is a bisector of $\angle DAC$.
\begin{figure}[h]
\centering
\includegraphics[height=1.7 in, keepaspectratio = true]{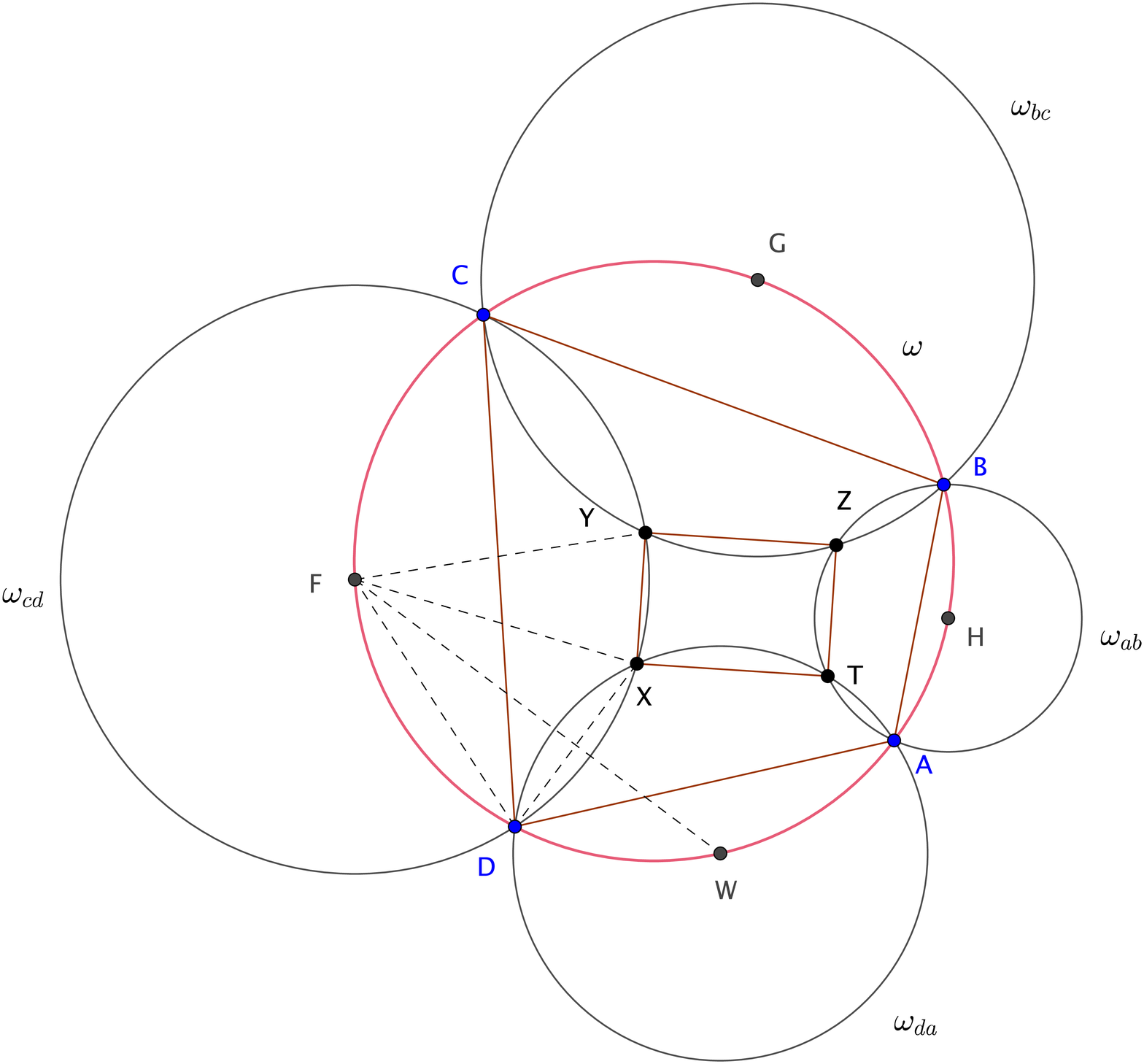}
\caption{Illustration for Lemma \ref{quad-out} (b)}
\label{quad-rect-fig}
\end{figure}
Similar argument shows that $W$, $X$ and $C$ are collinear and $CW$ is a bisector of $\angle DCA$. Thus, $X$ is the incenter of $\triangle ADC$. By the same token, we may conclude that $Y$, $Z$ and $T$ are the incenters of $\triangle DCB$, $\triangle CBA$ and $\triangle BAD$, respectively. Then, the result in Lemma \ref{quad-out} (b) follows from Problem 6.13 in \cite{prasolov}.
\end{proof}
{\it Remarks}: From Lemma \ref{fundament} it follows that inner part of a side disk of $C_{n}$ is always contained in $C_{n}$. Since outer parts of side disks of $C_{n}$ are disjoint, this implies that \emph{two side disks of $C_{n}$ with $n\geq 2$ intersect if and only if they intersect in the region enclosed by $C_{n}$.} To our knowledge, the only widely known result directly related to Lemma \ref{quad-out} is Miquel's four circles theorem which states that when $\omega_{ab}$, $\omega_{bc}$, $\omega_{cd}$ and $\omega_{da}$ are not necessarily centered on $\partial(\omega)$, $X, Y, Z, T$ are concyclic (see \cite{wells}; p.151). 

For $C_{n}$, let $d(C_{n})$ be the number of disjoint pairs among its side disks. Our main result is as follows.
\begin{theorem}
\label{circle-thm}
For $n\geq 3$, $\frac{(n-2)(n-3)}{2}\leq d(C_{n})\leq \frac{n(n-3)}{2}$.
\end{theorem}
\begin{proof} Since $d(C_{3})=0$, as all side disks are neighbouring, we assume $n\geq4$. We prove the lefthand inequality in three steps. 

\vspace{10pt}\underline{STEP 1}: Let us prove that $1\leq d(C_{4})$. 

\vspace{10pt}\noindent Let $A, B, C, D$ be the points marked on $C_{4}$ and $F, G, H, W$ be the centers of its four side disks. Further let $X, Y, Z, T$ be points other than $A, B, C, D$ in which pairs of neighbouring side disks intersect by their boundaries (see Fig. \ref{circle-fig}). By Lemma \ref{quad-out} (b), we know that $XYZT$ is a rectangle. Let $E=XZ\cap YT$, i.e. the intersection of the diagonals of $XYZT$. We claim that $E=WG\cap FH$. Since $|FY|=|FX|$, $|HZ|=|HT|$ and $XYZT$ is a rectangle, points $F$, $H$ and the mid-points of the sides $XY$ and $ZT$ are collinear. Similarly, $G$, $W$ and the midpoints of $ZY$ and $TX$ are collinear. Thus, $FH$ and $GW$ intersect in point which is intersection of two line segments connecting the midpoints of opposite sides of $XYZT$, which must be $E$. This proves our claim.

Since $\measuredangle XEY+\measuredangle YEZ=\pi$ one of these two angles (summands) must be at most $\frac{\pi}{2}$, and without loss of generality we may assume that $\measuredangle XEY\leq \frac{\pi}{2}$. Then, we claim that the side disks centered at $F$ and $H$ are disjoint. To see this, it suffices to prove that $E$ is located outside of these side disks, as then we have $|FE|>r_{F}$ and $|EH|>r_{H}$, hence, $|FH|=|FE|+|EH|>r_{F}+r_{H}$, where $r_{F}$ and $r_{H}$ are the radii of the disks to be shown as disjoint. Let us then prove that $E\notin F\text{-disk}$ and a similar argument shows that $E\notin H\text{-disk}$. By Lemma \ref{quad-out} (a), $XY$-line separates $F$ and rectangle $XYZT$. Since $E$ is an interior point of $XYZT$, we can conclude that $XY$-line strictly separates $E$ and $F$.

\begin{figure}[h]
\centering
\includegraphics[height=2.0 in, keepaspectratio = true]{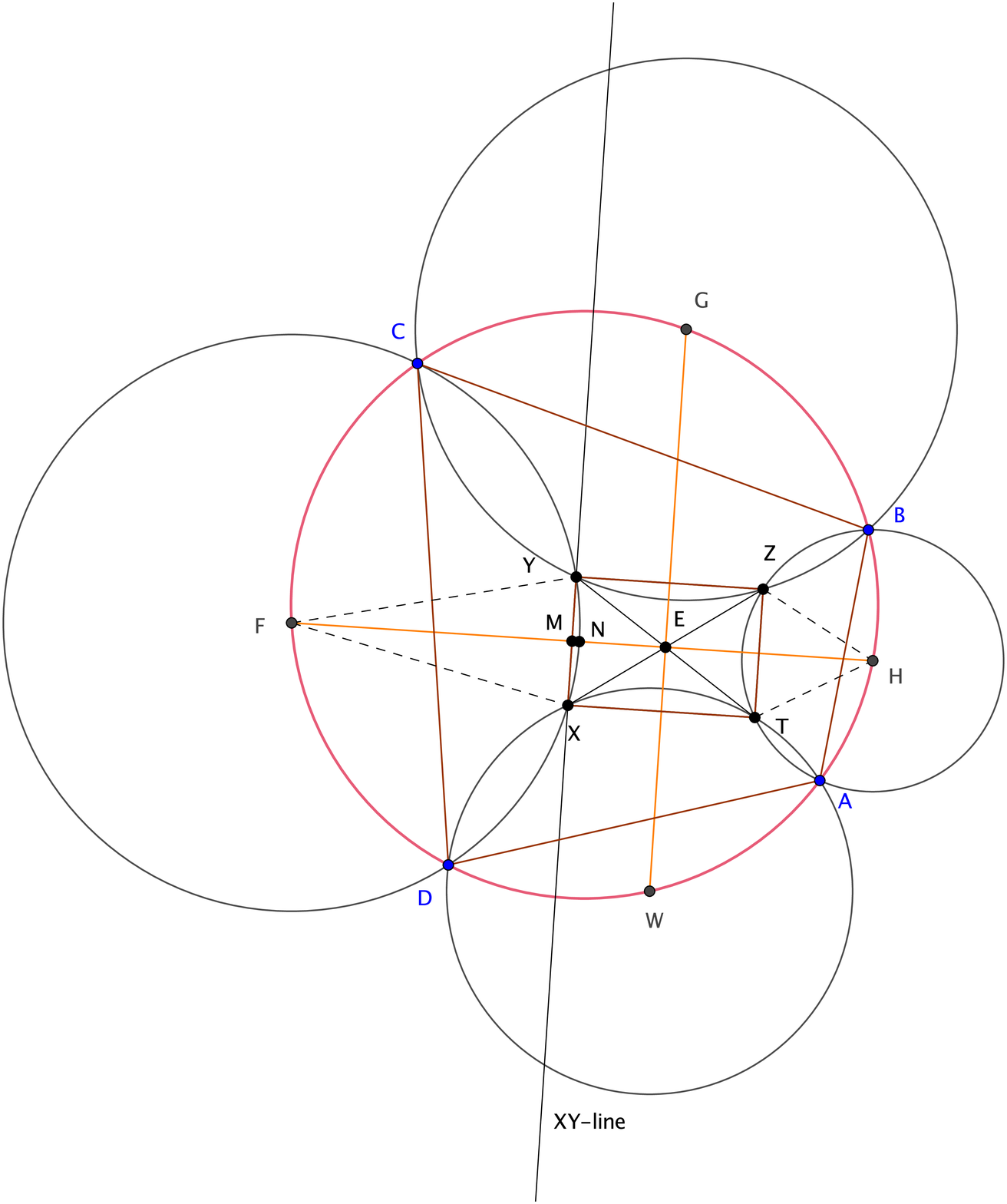}
\caption{Side disks of $C_{4}$}
\label{circle-fig}
\end{figure}

Let $M$ be the midpoint of $XY$ and $N=\partial(F\text{-disk})\cap FH$. It is clear that $N$ and $E$ lie on the same half-plane with respect to $XY$-line, as both are strictly separated from $F$ by the line. We know that $FH$ passes through the midpoints of $XY$ and  $ZT$, thus it must be orthogonal to $XY$. Recall that, both $E$ and $N$ lie on $FH$. Then, $\measuredangle XEY\leq \frac{\pi}{2}$ together with the observation that $\measuredangle XNY> \frac{\pi}{2}$ imply that $|ME|>|MN|$.\footnote{Take the circle centered at $F$. It is clear that $XY$ is strictly shorter than its diameter. Then, for any $N^{\star}$ lying on the minor arc connecting $X$ and $Y$, we have $\measuredangle XN^{\star}Y> \frac{\pi}{2}$.} Then, $|FE|=|FM|+|ME|>|FM|+|MN|=r_{F}$. Thus, $E$ is outside of $F$-disk. This proves our last claim and completes STEP 1.

\vspace{10pt}\underline{STEP 2}: Consider $C_{n}$ and its side disks, labeled as $\omega_{1}, ..., \omega_{n}$ in the clockwise direction. For any $i,j=1,2,...,n$, let $(\omega_{i}, \omega_{j})$ be the set of disks strictly between $\omega_{i}$ and $\omega_{j}$, in the clockwise direction. We shall prove that if $\omega_{i}$ and $\omega_{j}$ intersect, then any disk in $(\omega_{i}, \omega_{j})$ is disjoint from any one in $(\omega_{j}, \omega_{i})$.

\vspace{10pt}\noindent We can assume that $\omega_{i}$ and $\omega_{j}$ are non-neighbouring as the result is trivial otherwise. Let $\omega_{i}$, $\omega_{j}$ be side disks of $AB$-arc and $CD$-arc, respectively. Then, $AB$-arc and $CD$-arc are disjoint and we can also assume that $A,B,C,D$ are located subsequently in the clockwise order. Let $\omega_{bc}$ and $\omega_{da}$ be the side disks of $BC$-arc and $DA$-arc, respectively (see Fig. \ref{step2}). Then by STEP 1 there must be a disjoint fair among $\omega_{i}$, $\omega_{j}$, $\omega_{bc}$ and $\omega_{da}$. But since the former two intersect, it must the latter two which are disjoint. By Lemma \ref{fundament}, we know that when restricted to the region enclosed by $C_{n}$, $\omega_{bc}$ contains all disks in $(\omega_{i}, \omega_{j})$, and similarly, $\omega_{da}$ contains all disks in $(\omega_{j}, \omega_{i})$.  
\begin{figure}[h]
\centering
\includegraphics[height=1.6 in, keepaspectratio = true]{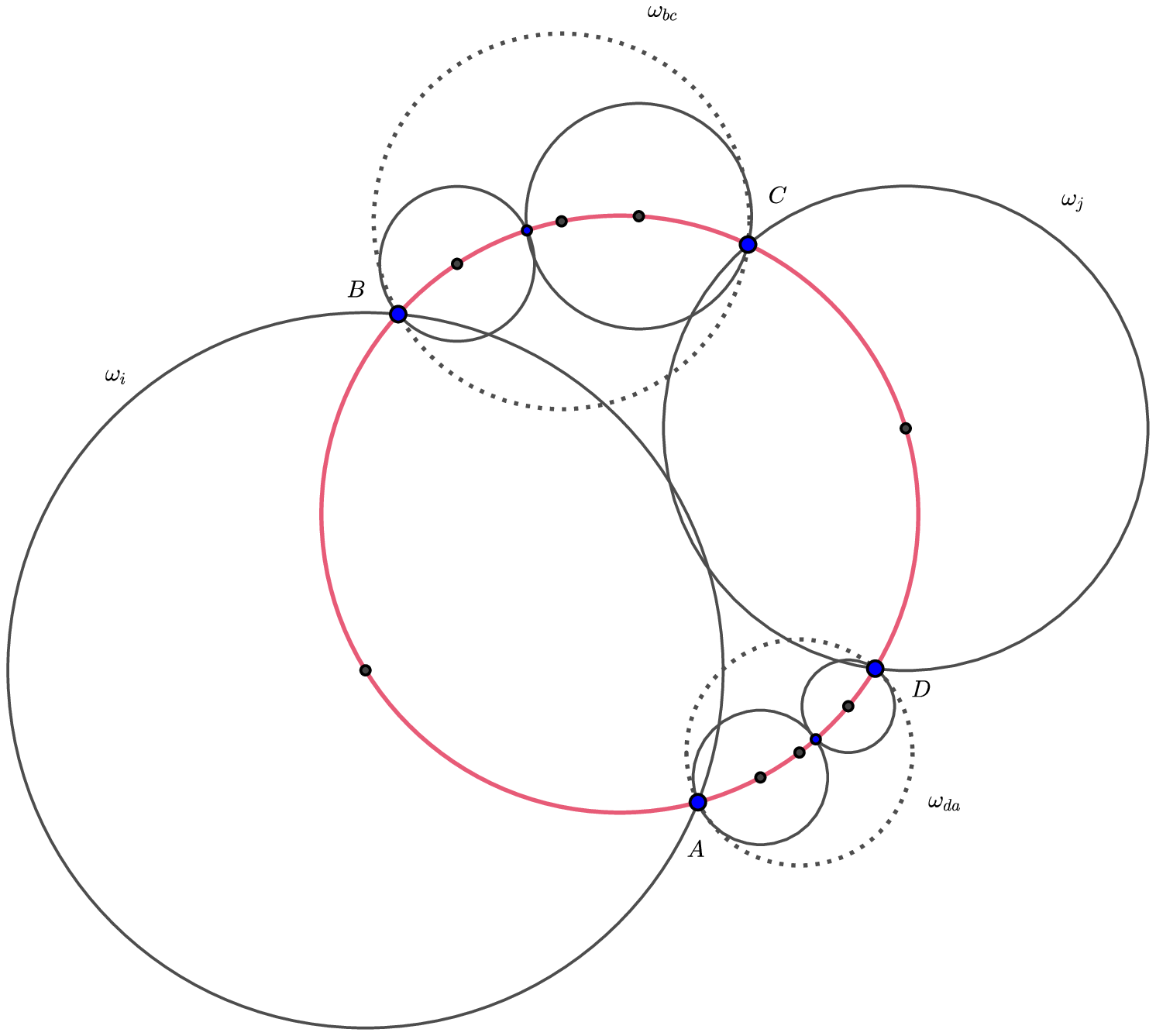}
\caption{$\omega_{i}$ and $\omega_{j}$ intersect}
\label{step2}
\end{figure}
This implies that, none of the disks in $(\omega_{i}, \omega_{j})$ intersects with a disk in $(\omega_{j}, \omega_{i})$ in the region enclosed by $C_{n}$. But since two side disks intersect only in that region, we can conclude that any disk in $(\omega_{i}, \omega_{j})$ is disjoint from any one in $(\omega_{j}, \omega_{i})$. This completes STEP 2.

\vspace{10pt}\underline{STEP 3}: Let us prove that for $n\geq 4$, $\frac{(n-2)(n-3)}{2}\leq d(C_{n})$.

\vspace{10pt}\noindent Take $P_{n}$, a convex $n$-gon, and label its vertices with the side disks of $C_{n}$ such that two disks of $C_{n}$ are neighbouring if and only if their associated vertices in $P_{n}$ are adjacent. Draw all $\frac{n(n-1)}{2}-n$ diagonals of $P_{n}$, and colour them with 
\begin{itemize}
\item Red if the side disks corresponding to the end vertices intersect, and
\item Blue if otherwise.
\end{itemize}
By STEP 2 we know that two red diagonals never cross in $P_{n}$. The maximal set of non-crossing diagonals of $P_{n}$ gives a triangulation of it, and every triangulation involves $n-3$ diagonals (see Theorem 1.8 in \cite{princeton}). Thus, the number of red diagonals is at most $n-3$, and the number of blue diagonals is at least $\frac{n(n-1)}{2}-n-(n-3)=\frac{(n-2)(n-3)}{2}$. This immediately implies that $\frac{(n-2)(n-3)}{2}\leq d(C_{n})$, and completes STEP 3. The lefthand inequality in Theorem \ref{circle-thm} is proved. Finally, since two neighbouring disks are never disjoint we have $d(C_{n})\leq \frac{n(n-3)}{2}$.
\end{proof}
{\it Remarks}: It is easy to show that the upper bound in Theorem \ref{circle-thm} is attained on $C_{n}^{\star}$, i.e. it is tight. Let $C_{n}^{\triangle}$ be a spherical great $n$-gon such that one of its side disks intersects with all the others, and any two of the other side disks intersect only if they are neighbouring. It is easy to show that this construction is well defined and $d(C_{n}^{\triangle})=\frac{(n-2)(n-3)}{2}$. Thus, the lower bound is also tight for $n\geq 3$.

We can now characterize the intersection graph of side disks of $C_{n}$.
\begin{theorem}
\label{planar}
The intersection graph of side disks of $C_{n}$ for $n\geq 3$ is a subgraph of a triangulation of a convex $n$-gon. In particular, it is outerplanar.
\end{theorem}
\begin{proof} Let $G(C_{n})$ be the intersection graph. The result is obvious when $n=3$. For $n\geq 4$, in STEP 2 of proof of Theorem \ref{circle-thm} we showed that $G(C_{n})$ can be drawn with no crossing edges. So, it is a subgraph of triangulation of the convex polygon with vertices at the centers of the side disks, hence outerplanar.
\end{proof}

\end{document}